\documentclass[a4paper]{amsart} 
\usepackage{amssymb,amscd,verbatim} 
\usepackage[all]{xy}
\SelectTips{cm}{}
\usepackage{tikz-cd}
\usepackage{spverbatim}
\usepackage{hyperref}
\calclayout

\title[Polynomial representations and Schur-Weyl duality]{Polynomial
  representations of $\GL(n)$ and Schur-Weyl duality}

\thanks{Version from November 4, 2013.}

\author{Henning Krause}
\address{Henning Krause\\ Fakult\"at f\"ur Mathematik\\
  Universit\"at Bielefeld\\ D-33501 Bielefeld\\ Germany.}
\email{hkrause@math.uni-bielefeld.de}

\newtheorem{lem}{Lemma}

\newtheorem{cor}[lem]{Corollary}
\newtheorem{thm}[lem]{Theorem}

\theoremstyle{remark}
\newtheorem{rem}[lem]{Remark}

\theoremstyle{definition}

\newtheorem{defn}[lem]{Definition}

\newcommand{\End}{\operatorname{End}\nolimits}

\newcommand{\GL}{\operatorname{GL}\nolimits}

\newcommand{\TS}{\operatorname{TS}\nolimits}

\newcommand{\lto}{\longrightarrow}

\def\a{\alpha}

\def\e{\varepsilon}

\def\p{\phi}

\def\s{\sigma}

\def\la{\lambda}

\def\bbF{{\mathbb F}}

\def\bbN{{\mathbb N}}

\def\bbQ{{\mathbb Q}}

\def\bbZ{{\mathbb Z}}

\begin{document}

\begin{abstract} 
  Polynomial representations of general linear
  groups and modules over Schur algebras are compared. We work over an
  arbitrary commutative ring and show that Schur-Weyl duality is the
  key for an equivalence between both categories.
\end{abstract}

\maketitle
\setcounter{tocdepth}{1}

Polynomial representations of general linear groups can be identified
with modules over Schur algebras. This follows from work of Schur
\cite{Sc1927} for the field of complex numbers and has been extended
to infinite fields by Green \cite{Gr1980}. The situation for finite
fields seems to be less clear. In fact, Benson and Doty pointed out
that Schur-Weyl duality may fail if the field is too small \cite{BD2009}.

For any commutative ring $k$ and any pair $n,d$ of natural numbers,
there is a canonical functor from modules over the Schur algebra
$S_k(n,d)$ to degree $d$ polynomial representations of $\GL_k(n)$.
Assuming that Schur-Weyl duality holds, we show that this functor is
an equivalence. The converse is true when $k$ is a field.

\begin{center}
*\quad *\quad *
\end{center}

Throughout we fix a commutative ring $k$. For a positive integer $n$
let $\GL_k(n)$ denote the group of invertible $n\times n$-matrices
over $k$. A \emph{representation} of  $\GL_k(n)$ over $k$ is a $k$-module
$V$ together with a group homomorphism
\[\rho\colon \GL_k(n)\lto \GL_k(V)\]
into the group of $k$-linear automorphisms of $V$. Given an integer
$d\ge 0$, we call a representation $\rho$ \emph{polynomial of degree
  $d$} if it can be lifted to a map
\[\hat\rho\colon \End_k(k^n)\lto \End_k(V)\]
which is homogeneous polynomial of degree $d$.

Recall from \cite[IV.5.9]{Bo1981} the definition of a polynomial
map. Given a pair of $k$-modules $M,N$ such that $M$ is free, a map
$f\colon M\to N$ is \emph{homogeneous polynomial of degree
  $d$} if there exists a basis $(x_i)_{i\in I}$ of $M$ and a family
$(y_\nu)_{\nu\in\bbN^{(I)},\, |\nu|=d}$ of elements in $N$ such that
\[f\Big(\sum_{i\in I}\la_ix_i\Big)=\sum _{\nu\in\bbN^{(I)},\, |\nu|=d}\la^\nu
y_\nu\] for all $(\la_i)\in k^{(I)}$, where  $|\nu|=\sum_{i\in
  I}\nu_i$ and $\la^\nu=\prod_{i\in
  I}\la_i^{\nu_i}$.

The definition of a polynomial representation coincides with the usual
one\footnote{Given a basis $(v_i)_{i\in I}$ of $V$, there are
  homogeneous polynomials $f_{ij}\in k[X_{rs}]$ of degree $d$ in $n^2$
  indeterminates such that $\rho (\a)(v_i)=\sum_{j\in I}f_{ij}(\a_{rs})
  v_j$ for each $\a=(\a_{rs})$ in $\GL_k(n)$.}  when $k$ is a field
(cf.\ \cite[2.2]{Gr1980} or \cite[Appendix~A.8]{MacD1995}).

Set $E=k^n$. The symmetric group $\mathfrak S_d$ acts on $E^{\otimes
d}=E\otimes_k\ldots \otimes_k E$ via 
\begin{equation}\label{eq:action}
\s(x_1\otimes\ldots\otimes x_d)=x_{\s(1)} \otimes\ldots\otimes
x_{\s(d)}\tag{$*$}
\end{equation}
and this induces a conjugation action on the endomorphisms of
$E^{\otimes d}$. The \emph{Schur algebra} is by definition the algebra
\[S_k(n,d)=\End_k(E^{\otimes d})^{\mathfrak S_d}\] of
$\mathfrak S_d$-invariant endomorphisms \cite[2.3]{Gr1980}. The
diagonal action of $\GL_k(n)$ on $E^{\otimes d}$ yields an embedding
$\GL_k(n)\to \End_k(E^{\otimes
  d})$ which extends to a $k$-algebra
homomorphism \[\p\colon k\GL_k(n)\lto S_k(n,d).\]

Restriction of scalars via $\p$ induces a functor from modules over
$S_k(n,d)$ to modules over $k\GL_k(n)$.

\begin{lem}\label{le:res}
  A module over $S_k(n,d)$ yields via restriction of scalars a
  representation of $\GL_k(n)$ over $k$ which is polynomial of degree
  $d$.

  Conversely, a degree $d$ polynomial representation $\rho\colon
  \GL_k(n)\to \End_k(V)$ yields a linear action of $S_k(n,d)$ on $V$
  which restricts to $\rho$ provided there exists a polynomial lifting
  $\hat\rho\colon \End_k(k^n)\to \End_k(V)$ which can be extended to a
  $k$-algebra homomorphism $S_k(n,d)\to\End_k(V)$.
\end{lem}

The proof  requires some preparation. For a $k$-module
$M$ and an integer $d\ge 0$ let 
\[\TS^d(M)=\{x\in M^{\otimes d}\mid \s x=x\text{ for all
}\s\in\mathfrak S_d\}\] denote the $k$-module of \emph{degree $d$
  symmetric tensors}, where $\mathfrak S_d$ acts on $M^{\otimes d}$
via \eqref{eq:action}. Each $x\in M$ yields an element
$\gamma_d(x)=x\otimes\ldots\otimes x$ in $\TS^d(M)$.

The next lemma is  useful since  \[S_k(n,d)\cong
  \TS^d(\End_k(k^n)).\]

\begin{lem}[{\cite[IV.5, Proposition~13]{Bo1981}}]\label{le:sym} 
Let $M,N$ be
$k$-modules such that $M$ is free. Then a map $f\colon
  M\to N$ is homogeneous polynomial of degree $d$ if and only if there exists a
  $k$-linear map $h\colon\TS^d(M)\to N$ such that
  $f(x)=h(\gamma_d(x))$ for all $x\in M$. \qed
\end{lem}

\begin{proof}[Proof of Lemma~\ref{le:res}] Use
  Lemma~\ref{le:sym}. \end{proof}

The following theorem clarifies the relation between degree $d$
polynomial representations of $\GL_k(n)$ and modules over the Schur
algebra $S_k(n,d)$. This requires another definition.

\begin{defn}\label{de:epi}
Let $\p\colon A\to B$ be a $k$-algebra homomorphism. We call $\p$ a
\emph{strong epimorphism} of $k$-algebras if  the following holds.
\begin{enumerate}
\item The map $\p$ is an epimorphism of $k$-algebras, that is,
  given $k$-algebra homomorphisms $\psi_1,\psi_2\colon B\to C$ such
  that $\psi_1\p=\psi_2\p$, then $\psi_1=\psi_2$. \item An
  $A$-module $V$  is the restriction of a $B$-module if the
  canonical map $A\to \End_k(V)$ factors through $\p$ via a $k$-linear
  map $B\to \End_k(V)$. \end{enumerate} \end{defn}

Clearly, every surjective algebra homomorphism is a strong
epimorphism. However, the example $\bbZ\to \bbQ$ shows that the
converse is not true. 

\begin{thm}\label{th:main}
  The canonical functor from modules over the Schur algebra $S_k(n,d)$ to degree
  $d$ polynomial representations of $\GL_k(n)$ is an equivalence if
  and only if the $k$-algebra homomorphism $\p\colon k\GL_k(n)\to
  S_k(n,d)$ is a strong epimorphism.
\end{thm}

\begin{rem} 
  (1) Schur-Weyl duality asserts that the canonical map $\p$ is
  surjective.  For example, this holds when $k$ is an infinite field (cf.\ (2.4b)
  and (2.6c) in \cite{Gr1980}) or a finite field having more than $d$
  elements (cf.\ Theorem~4.3 in \cite{BD2009}).

  (2) Benson and Doty \cite{BD2009} noticed that Schur-Weyl duality
  fails for $k=\bbF_2$ and $n=2=d$. More precisely, the canonical map
  $\p$ is not an epimorphism in this case. Let $k[\e]$ denote the
  algebra of dual numbers ($\e^2=0$).  We have
  \[\bbF_2\GL_{\bbF_2}(2)\cong \bbF_2 [\e] \times \bbF_2 [\e] \times
  \bbF_2 [\e] \] while $S_{\bbF_2}(2,2)$ is Morita equivalent to the
  Auslander algebra \cite[III.4]{Au1971} of $\bbF_2 [\e]$. It
  follows that modules over the Schur algebra cannot embed fully
  faithfully into the category of modules over
  $\bbF_2\GL_{\bbF_2}(2)$.

(3) The failure of Schur-Weyl duality for $k=\bbF_2$ implies that the
  canonical map
\[\bbZ\GL_\bbZ(n)\lto S_\bbZ(n,d)\] is not
surjective for $n=2=d$, since $S_k(n,d)\cong S_\bbZ(n,d)\otimes_\bbZ k$.
\end{rem}

\begin{proof}[Proof of Theorem~\ref{th:main}]
From Lemma~\ref{le:res} we know that the restriction of a
  $S_k(n,d)$-module is a polynomial representation of degree $d$.

  Restriction via $\p$ induces a fully faithful functor from
  modules over $S_k(n,d)$ to modules over $k\GL_k(n)$ if and only if
  $\p$ is an
  epimorphism of rings; this is well-known \cite[XI.1]{St1975}. 

  Let us consider the condition (2) in Definition~\ref{de:epi}. We
  verify this, assuming that every degee
  $d$ polynomial representation is in the image of the restriction
  functor. To
  this end fix a $k\GL_k(n)$-module $V$ such that the canonical map
  $k\GL_k(n)\to \End_k(V)$ factors through $\p$ via a $k$-linear
  map $ S_k(n,d)\to \End_k(V)$. Applying Lemma~\ref{le:sym}, this means that $V$ is a
  degree $d$ polynomial representation of $\GL_k(n)$ and therefore the
  restriction of a $S_k(n,d)$-module.

  Now suppose that $\p$ satisfies condition (2)  in
  Definition~\ref{de:epi} and fix a degree $d$
  polynomial representation $\rho\colon \GL_k(n)\to\End_k(V)$.
  It follows from Lemma~\ref{le:sym} that the lifting
  $\hat\rho\colon \End_k(k^n)\to \End_k(V)$ yields a $k$-linear map
  $S_k(n,d)\to\End_k(V)$. Thus the map $k\GL_k(n)\to \End_k(V)$ extending $\rho$
  factors through $\p$. The assumption
  on $\p$ implies that the representation $\rho$ lies in the image
  of the restriction functor.
\end{proof}

The statement of Theorem~\ref{th:main} can be strengthened when $k$ is
field. This follows from the  lemma below.

\begin{lem}
  Let $k$ be a field. Then a $k$-algebra homomorphism $\p\colon A\to
  B$ is a strong epimorphism if and only if $\p$ is surjective.
\end{lem}
\begin{proof}
  One direction is clear. So suppose that $\p$ is a strong epimorphism
  and denote by $A'$ is image. Given a $k$-module $V$, every
  $k$-algebra homomorphism $A'\to \End_k(V)$ can be extended to a
  $k$-linear map $B\to\End_k(V)$ since $k$ is field. Thus restriction
  along the inclusion $A'\to B$ induces an equivalence between modules
  over $B$ and modules over $A'$. It follows that $A'=B$.
\end{proof}

\begin{cor}
Let $k$ be a field. Then the following are equivalent.
\begin{enumerate}
\item  Schur-Weyl duality holds, that is,  $\p\colon k\GL_k(n)\to
  S_k(n,d)$ is surjective.
\item The category of modules over the Schur algebra $S_k(n,d)$ is
  equivalent (via $\p$) to the category of degree $d$ polynomial
  representations of $\GL_k(n)$.\qed
\end{enumerate}
\end{cor}

\begin{center}
*\quad *\quad *
\end{center}

I am grateful to Steve Donkin, Steve Doty, and Andrew Hubery for helpful comments.

\end{document}